\documentclass[11pt]{amsart}
\usepackage{verbatim,hyperref}
\usepackage{rotating} 
\usepackage{amssymb}
\usepackage{url}

\newtheorem{theorem}{Theorem}[section]

\newtheorem{lemma}[theorem]{Lemma}
\newtheorem{fact}[theorem]{Fact}
\newtheorem{cor}[theorem]{Corollary}

\newtheorem{definition}[theorem]{Definition}

\newtheorem{conjecture}[theorem]{Conjecture}

\theoremstyle{plain}
\numberwithin{equation}{theorem}

\theoremstyle{remark}
\newtheorem{claim}[theorem]{Claim}

\newcommand{\Q}{{\mathbb Q}}

\newcommand{\Z}{{\mathbb Z}}

\newcommand{\cS}{{\mathcal S}}

\newcommand{\OO}{{\mathcal O}}

\newcommand{\id}{\operatorname{id}}

\newcommand{\Kbar}{\overline K}

\def\Frac{\operatorname{Frac}}

\newcommand{\Qbar}{\bar{\Q}}

\DeclareMathOperator{\Gal}{Gal}
\DeclareMathOperator{\GL}{GL}

\DeclareMathOperator{\N}{\mathbb{N}}

\DeclareMathOperator{\tor}{tor}

\DeclareMathOperator{\End}{End}

\newcommand{\isomto}{\overset{\sim}{\rightarrow}}

\newcommand{\bP}{{\mathbb P}}

\newcommand{\bA}{{\mathbb A}}

\newcommand{\lra}{\longrightarrow}

\newcommand{\cU}{\mathcal{U}}

\newcommand{\cL}{\mathcal{L}}

\begin{document}

\title{Density of orbits of endomorphisms of abelian varieties}

\author{Dragos Ghioca}
\author{Thomas Scanlon} 

\thanks{D.G. is partially supported by an NSERC grant. T.S. is partially supported by NSF Grant DMS-1363372.
This material is based upon work supported by the National Science Foundation under 
Grant No. 0932078 000 while the authors
were in residence at the Mathematical Sciences Research Institute in Berkeley, 
California, during the Spring 2014 semester}

\address{Department of Mathematics \\
University of British Columbia\\
Vancouver, BC V6T 1Z2\\
Canada}
\email{dghioca@math.ubc.ca}

\address{Mathematics Department \\ University of California Berkeley\\
  Evans Hall \\ Berkeley CA, 94720-3840
}
\email{scanlon@math.berkeley.edu}


\begin{abstract} Let $A$ be an abelian variety defined over $\Qbar$, and let $\varphi$ be a 
 dominant endomorphism of $A$ as an algebraic variety. We prove that either there exists
  a non-constant rational fibration preserved by $\varphi$, or there exists a point 
  $x\in A(\Qbar)$ whose $\varphi$-orbit is Zariski dense in $A$. This provides a 
  positive answer for abelian varieties of a question raised by Medvedev and the 
  second author in \cite{medvedev-scanlon}. We prove also a stronger statement of this result in which $\varphi$ is replaced by any commutative finitely generated monoid of dominant endomorphisms of $A$.
  
\end{abstract}

\maketitle


\section{Introduction}
\label{intro section}

The following conjecture was raised in \cite[Conjecture~7.14]{medvedev-scanlon} (motivated by a conjecture of Zhang \cite{Zhang-lec} for polarizable endomorphisms of projective varieties).
\begin{conjecture}
\label{M-S conjecture}
Let $K$ be a field of characteristic $0$, let $\Kbar$ be an algebraic closure 
of $K$, let $X$ be an irreducible algebraic variety defined over $\Kbar$, 
and let $\varphi:X\lra X$ be a dominant rational self-map. We suppose there exists 
no positive dimensional algebraic variety $Y$ and dominant rational 
map $f:X\lra Y$ such that $f\circ \varphi = f$. Then there 
exists $x\in X(\Kbar)$ whose forward $\varphi$-orbit is Zariski dense in $X$.
\end{conjecture}

We denote by $\OO_\varphi(x)$ the forward $\varphi$-orbit, 
i.e. the set of all $\varphi^n(x)$ for $n\ge 0$, where by $\varphi^n$ we 
denote the $n$-th compositional power of $\varphi$. 
Conjecture~\ref{M-S conjecture} was proven in \cite[Theorem~7.16]{medvedev-scanlon} 
in the special case $X=\bA^m$, and $\varphi:=(f_1,\dots, f_m)$ is given by the 
coordinatewise action of $m$ one-variable polynomials $f_i$. In this paper 
we prove Conjecture~\ref{M-S conjecture} when $X$ is an abelian variety. 
This is the fourth known case of Conjecture~\ref{M-S conjecture} 
(besides the case  proven by Medvedev and the second author 
in~\cite{medvedev-scanlon}, Amerik, Bogomolov and 
Rovinsky~\cite{Amerik-Bogomolov-Rovinsky} exploited the local dynamical 
behaviour of the map $\varphi$ to prove a special case of 
Conjecture~\ref{M-S conjecture} assuming there is a \emph{good} $p$-adic analytic 
parametrization for the orbit $\OO_\varphi(x)$, and recently, the case when $X$ is a surface was proven in \cite{BGT} using also $p$-adic methods).

\begin{theorem}
\label{abelian varieties}
Let $K$ be a field of characteristic $0$, and let $\Kbar$ be a fixed algebraic closure of $K$. 
Let $A$ be an abelian variety defined over $\Kbar$, and let $\sigma:A\lra A$ be a finite map
of agebraic varieties. Then the following statements are equivalent:
\begin{enumerate}
\item[(1)] there exists $x\in A(\Kbar)$ such that $\OO_\sigma(x)$ is Zariski dense in $A$.
\item[(2)]  there exists no non-constant rational map $f:A\lra \bP^1$ such 
that $f\circ \sigma =f$.
\end{enumerate}
\end{theorem}

The motivation for Conjecture~\ref{M-S conjecture} comes from two different directions. First, 
Zhang \cite[Conjecture~4.1.6]{Zhang-lec} proposed a variant of Conjecture~\ref{M-S conjecture} 
for polarizable endomorphisms $\varphi$ of projective varietieties $X$ defined over $\Qbar$ (we 
say that $\varphi$ is \emph{polarizable} if there exists an ample line bundle $\cL$ on $X$ so 
that $\varphi^*(\cL)\isomto \cL^{\otimes d}$ for some integer $d>1$). The polarizability 
condition imposed by Zhang is stronger than the hypothesis from Conjecture~\ref{M-S conjecture} 
that $\varphi$ preserves no non-constant fibration of $X$. The motivation for the stronger 
hypothesis appearing in  \cite[Conjecture~4.1.6]{Zhang-lec} lies in the fact that in his 
seminal paper \cite{Zhang-lec}, Zhang was interested in the arithmetic properties exhibited by 
the dynamics of endomorphisms of projective varieties. In particular, Zhang was interested in formulating good dynamical analogues of the classical Manin-Mumford and Bogomolov 
Conjectures, and thus he wanted to use the canonical heights associated to polarizable endomorphisms (previously introduced by Call and Silverman \cite{Call-Silverman}). The second motivation for Conjecture~\ref{M-S conjecture} comes from the fact that 
its conclusion is  known assuming $K$ is an uncountable field of characteristic $0$ 
(see~\cite{Amerik-Campana}). More precisely, in \cite{Amerik-Campana}, Amerik and Campana 
proved that if $\varphi$ preserves no non-constant rational fibration, then there exist 
countably many proper subvarieties $Y_i$ of $X$ so that for 
each $x\in X(\Kbar)\setminus \cup_i Y_i(\Kbar)$, the orbit $\OO_\varphi(x)$ is Zariski dense in 
$X$. However, if $K$ is countable, then the result of Amerik and Campana leaves open the 
possibility that each algebraic point of $X$ is also an algebraic point of some 
subvariety $Y_i$ for some positive integer $i$. Hence, Conjecture~\ref{M-S conjecture} raises a 
deeper arithmetical question.

We are able to extend Theorem~\ref{abelian varieties} to the action of any commutative finitely 
generated monoid of dominant endomorphisms of an abelian variety. For a monoid $S$ of 
endomorphisms of an abelian variety $A$, and for any point $x\in A$, we let $\OO_S(x)$ be 
the $S$-orbit of $x$, i.e. the set of all   $\psi(x)$, where $\psi\in S$. 

\begin{theorem}
\label{general theorem}
Let $K$ be a field of characteristic $0$, let $\Kbar$ be an algebraic closure of $K$, and let  
$S$ be a finitely generated, commutative monoid of dominant endomorphisms of an abelian variety 
$A$ defined over $\Kbar$. Then either there exists $x\in A(\Kbar)$ such that $\OO_S(x)$ is 
Zariski dense in $A$  or there exists a non-constant rational map $f:A\lra \bP^1$ such that 
$f\circ \sigma = f$ for each $\sigma\in S$.
\end{theorem}

It is reasonable to formulate an extension of Conjecture~\ref{M-S conjecture} to the setting of 
a monoid action of rational self-maps on an algebraic variety $X$. However, there are several 
additional complications arising from such a generalization even in the case of the dynamics of 
endomorphisms of an abelian variety $A$, such as:
\begin{enumerate}
\item[(i)] Should we impose any restriction on the monoid $S$? Theorem~\ref{general theorem} is 
valid only for finitely generated, commutative monoids, and our method of proof does not seem 
to extend beyond this case (at least not in the case of arbitrary endomorphisms of an abelian 
variety $A$; if $S$ is an arbitrary commuting monoid of dominant group endomorphisms of $A$, 
then the conclusion of Theorem~\ref{general theorem} holds easily).
\item[(ii)] Assuming there is no non-constant fibration preserved by the entire monoid $S$, is 
it true that there exists some $\sigma\in S$ and there exists $x\in A(\Kbar)$ such that 
$\OO_\sigma(x)$ is Zariski dense in $A$? We have examples of  non-commuting monoids $S$ 
generated by two group homomorphisms of $A$ such that there is no non-constant fibration 
preserved by $S$, even though for \emph{each} $\sigma\in S$ there exists a non-constant 
fibration preserved by $\sigma$. On the other hand, if $S$ is a commutative monoid of group 
homomorphisms of $A$, then it is easy to see that the above question has a positive answer.
\end{enumerate}
Finally, we note that Amerik-Campana's result \cite{Amerik-Campana} was extended 
in~\cite{preprint} for arbitrary monoids $S$ acting on an algebraic variety $X$ through 
dominant rational endomorphisms, i.e. if there is no non-constant rational fibration preserved 
by $S$, then there exist countably many proper subvarieties $Y_i\subset X$ such that for each 
$x\in X(\Kbar)\setminus \cup_i Y_i(\Kbar)$, the orbit $\OO_S(x)$ is Zariski dense in $X$. 
Again, similar to \cite{Amerik-Campana}, the result of \cite{preprint} leaves open the 
possiblity that if $K$ is countable, then $X(\Kbar)$ may be covered by $\cup_i Y_i(\Kbar)$.

Here is the strategy for our proof. By the classical theory of abelian varieties, we know that each endomorphism $\varphi$ of an abelian variety $A$ is of the form $T_y\circ \tau$ (for a translation map $T_y$, with $y\in A$) and some (algebraic) group homomorphism $\tau$. Since the endomorphisms $\varphi$ from the given monoid $S$ commute with each other, we obtain that also the corresponding group homomorphisms $\tau$ commute with each other. This gives us a lot of control on the action of the corresponding group homomorphisms $\tau$; in particular, if all endomorphisms from $S$ would also be group homomorphisms, then Theorem~\ref{general theorem} would follow easily. Essentially, in that special case, the problem would reduce to the following dichotomy: either there exists a positive dimensional algebraic subgroup of $A$ which is fixed by a finite index submonoid of $S$, or there exists a \emph{single} element $\sigma$ of $S$, and there exists an algebraic point $x$ of $A$ whose $\sigma$-orbit is Zariski dense in $A$ (essentially, such a point $x$ has the property that the cyclic subgroup generated by $x$ is Zariski dense in $A$). So, if $S$ consists only of group homomorphisms, the conclusion of Theorem~\ref{general theorem} holds even in a stronger form. However, if the endomorphisms from $S$ are not all group endomorphisms of $A$, then the proof is much more complicated. One can still find a necessary and sufficient condition under which there exists a non-constant rational fibration preserved by all elements in $S$, but that condition is very technical. Finally, we note that the exact same proof works to prove a variant of Theorem~\ref{general theorem} with the abelian variety $A$ replaced by a power of the torus. On the other hand, our proof does not seem to generalize to the case of semiabelian varieties due to the failure of the Poincar\'e Reducibility Theorem (see Fact~\ref{Poincare}) for semiabelian varieties which are not isogenuous to split semiabelian varieties.

The plan of the paper is as follows. In Section~\ref{section monoids} we note several easy 
statements regarding monoids. We continue by stating some basic facts about abelian varieties 
in Section~\ref{abelian varieties section}. Then, in Section~\ref{reductions} and then in Section~\ref{necessary lemmas} we prove various 
reductions of the Theorem~\ref{general theorem}, respectively some 
auxilliary results needed later. In Section~\ref{one section} we prove Theorem~\ref{abelian 
varieties} as a way to introduce the reader to the more elaborate argument needed for the proof 
of Theorem~\ref{general theorem} (which is completed in Section~\ref{many section}). While
Theorem~\ref{abelian varieties} is a special case of Theorem~\ref{general theorem}, 
we have chosen to prove them separately because we believe it is easier for the reader to read 
first the argument done for a cyclic monoid (Theorem~\ref{abelian varieties}), which avoids 
some of the technicalities appearing in the proof of Theorem~\ref{general theorem}.

\medskip

{\bf Acknowledgments.} We thank Jason Bell, Thomas Tucker 
and Zinovy Reichstein for several useful discussions 
while writing this paper.  


\section{General results regarding monoids}
\label{section monoids}

We  need some basic facts about finitely generated, commutative monoids. First we need a 
definition.
\begin{definition}
\label{barT}
Let $S$ be any finitely generated, commutative monoid. For each submonoid $T\subseteq S$, we 
denote by $\bar{T}$ the submonoid containing all $x\in S$ with the property that there exist 
$y,z\in T$ such that $xy=z$.
\end{definition}

We also recall that a monoid $S$ is called \emph{cancellative} if whenever $xy=xz$ for $x,y,z\in S$, then $y=z$. We note that a monoid of dominant endomorphisms of a given algebraic variety is a cancellative monoid.

\begin{lemma}
\label{T bar T}
Let $S$ be a cancellative, commutative monoid generated by the elements $\gamma_1,\dots, \gamma_s$, and let 
$T$ be a submonoid of $S$ such that $\bar{T}=S$. Then there exists a  finitely generated 
submonoid $T_0\subset T$ and there exists a positive integer $n$ such  that 
$\gamma_i^n\in \bar{T_0}$ for each $i=1,\dots, s$.
\end{lemma}

\begin{proof}
Let $f:\N^s\lra S$ be the homomorphism of monoids given by $f(e_i)=\gamma_i$, where $e_i\in\N^s$ 
is the $s$-tuple consisting only of zeros with the exception of the $i$-th entry which equals 
$1$. Let $U$ be the set of all $a\in \N^s$ such that $f(a)\in T$, and let $H$ be the subgroup 
of $\Z^s$ generated by $U$. Since $\bar{T}=S$, then $H=\Z^s$.  Therefore there exist $s$ 
linearly independent tuples in $U$; call them $u_1,\dots, u_s$. We claim that the monoid $T_0$ 
spanned by $f(u_1), \dots, f(u_s)$ satisfies the conclusion of our Lemma. 

Indeed, we first show that $\bar{T_0}=f(H_0\cap \N^s)$ where $H_0$ is the subgroup of $\Z^s$ generated by 
$u_1,\dots, u_s$. To see this, on one hand, it is clear that $\bar{T_0}\subseteq f(H_0\cap \N^s)$. Now, to see the reverse inclusion, note that $\bar{T_0}$ satisfies  $\bar{\bar{T_0}}=\bar{T_0}$. Indeed, if $x_1,x_2\in \bar{T_0}$ and $x\in S$ such that $xx_1=x_2$, we show that $x\in\bar{T_0}$. We have that there exist $y_i,z_i\in T_0$ such that $x_iy_i=z_i$ for $i=1,2$. Then we claim that $x(y_2z_1)=y_1z_2$ which would indeed show that $x\in\bar{T_0}$ because $y_2z_1,y_1z_2\in T_0$ (note that $T_0$ is a submonoid). To see the above equality in the cancellative monoid $S$, it suffices to prove that $x_1xy_2z_1=x_1y_1z_2$. Using that $x_1x=x_2$, $x_2y_2=z_2$ and $x_1y_1=z_1$, and that $S$ is commutative, we obtain the desired equality; hence $\bar{\bar{T_0}}=\bar{T_0}$ and thus $\bar{T_0}=f(H_0\cap \N^s)$.

Now, since $u_1,\dots, u_s$ are linearly independent over $\Z$ (as elements of 
$\Z^s$), then $H_0$ has finite index in $\Z^s$. So, there exists a positive integer $n$ such 
that $ne_i\in H_0$ for each $i=1,\dots, s$, and therefore $f(ne_i)=\gamma_i^n\in \bar{T_0}$.
\end{proof}

We also need some simple results from linear algebra.  The first is a consequence of the 
Lie-Kolchin triangularization theorem~\cite{kolchin}.

\begin{fact}
\label{commuting monoid linear algebra}
Let $S_0$ be a finitely generated, commuting monoid of matrices with entries in $\Qbar$. Then there exists an invertible matrix $C$ (with entries in $\Qbar$)  such that for each $A\in S_0$, the matrix $C^{-1}AC$ is upper triangular.
\end{fact}


Fact~\ref{commuting monoid linear algebra} will be used repeatedly throughout our proof. An 
important consequence of it is that the eigenvalues of each matrix in a commuting monoid $S_0$ 
are simply the entries on the diagonal (after a suitable change of coordinates). In particular, 
this has the following easy corollaries.

\begin{lemma}
\label{no root of unity}
Let $S_0$ be a commuting monoid of matrices with entries in $\Qbar$, generated by matrices $A_1,\dots, A_s$. Then there 
exists a positive integer $n$ such that for each matrix $A$ contained in the submonoid of $S_0$ 
generated by $A_1^n,\dots, A_s^n$, if $\lambda$ is an eigenvalue of $A$ which is also a root of 
unity, then $\lambda=1$.   
\end{lemma}

\begin{proof}
The conclusion holds with $n$ being the cardinality of the group of roots of unity contained in the number field $L$ which is generated by all the eigenvalues of the matrices $A_i$.
\end{proof}

\begin{lemma}
\label{minimal eigenspace for 1 linear algebra}
Let $S_0$ be a finitely generated, commuting monoid of matrices with the property that for each matrix $A$ in $S_0$, if $\lambda$ is an eigenvalue of $A$ which is a root of unity, then $\lambda=1$. Let $\cU_0$ be the set of 
matrices in $S_0$ with the property that the eigenspace  corresponding to the eigenvalue $1$ has 
the smallest dimension among all the matrices in $S_0$. Let $U_0$ be the submonoid generated by 
$\cU_0$. Then $\bar{U_0}=S_0$.
\end{lemma}

\begin{proof}
Using Fact~\ref{commuting monoid linear algebra}, we can choose a basis so that each matrix in $S$ is represented by an upper triangular matrix. Furthermore, we may assume each matrix in $\cU$ has the first $r$ entries on the diagonal equal to $1$, and none of the other entries on the diagonal are equal to $1$ (or to a root of unity). Indeed, we know each matrix in $\cU$ has $r$ entries on the diagonal equal to $1$; if these entries equal to $1$ would not be in the same places of the diagonal for two distinct matrices $A$ and $B$ in $\cU$, then for some positive integers $m$ and $n$ we would have that $A^mB^n$ has fewer than $r$  entries equal to $1$ on the diagonal. So, indeed the $r$ entries equal to $1$ appear in the same position on the diagonal for each matrix in $\cU$; so we may assume they are the first $r$ entries, while the remaining $\ell-r$ entries on the diagonal of each matrix in $\cU$ is not a root of unity. 

Let $A\in \cU$.  
Now, for each matrix $B\in S$, even if there exist entries in the positions $i=r+1,\dots, \ell$ on the diagonal which are equal to $1$, there exists a positive integer $n$ such that the entries on the diagonal of $A^nB$ in the positions $i=r+1,\dots, \ell$ are not equal to $1$. This completes our proof.
\end{proof}


\section{Abelian varieties}
\label{abelian varieties section}

First we recall several results regarding abelian varieties (see \cite{Milne} for more details). The setup will be as follows: $A$ is an abelian variety defined over a field 
$K$ of characteristic $0$; since one needs only finitely many parameters in order to define $A$, then 
we may assume $K$ is a finitely generated extension of $\Q$. We let $\Kbar$ be a fixed algebraic 
closure of $K$.  At the expense of replacing $K$ by a finite extension we may assume that all 
algebraic group endomorphisms of $A$ are defined over $K$; we denote by $\End(A)$ the ring  of 
all these endomorphisms. Since the torsion subgroup $C_{\tor}$ of any algebraic subgroup 
$C\subseteq A$ is Zariski dense in $C$, we conclude that any algebraic subgroup of $A$ is 
defined over $K(A_{\tor})$.  Frequently we will use the following facts.

\begin{fact}
\label{connected fact}
Let $B$ and $C$ be algebraic subgroups of the abelian variety $A$. Then $(B+C)^0 = \left(B^0+C^0\right)$, where for any algebraic subgroup $H\subseteq A$, we denote by $H^0$ the connected component of $H$ (containing $0\in A$).
\end{fact}

\begin{proof}
The algebraic group $B^0 + C^0$ is the image of the connected group $B^0 \times C^0$ under the sum map and is therefore connected.  As $B^0 \times C^0$ has finite index in $B \times C$,
its image under the sum map has finite index in $B + C$.  Hence, $B^0 + C^0 = (B+C)^0$.
\end{proof}


The following result is proven in \cite[Proposition~10.1]{Milne}.
\begin{fact}[Poincar\'e's Reducibility Theorem]
\label{Poincare}
If $B\subseteq A$ is an abelian subvariety of $A$, then there exists an abelian subvariety $C\subseteq A$ such that $A=B+C$ and $B\cap C$ is finite; in particular $A/B$ and $C$ are isogenous.
\end{fact}

Poincar\'e's Reducibility Theorem yields that any abelian variety is isogenous with a direct product of finitely many simple abelian varieties, i.e. $A\isomto A_0:=\prod_{i=1}^r C_i^{k_i}$, where each $C_i$ is simple. Then $\End(A)\isomto\End(A_0)$ (see also \cite[Section~1.10]{Milne}), and moreover $\End(A_0)\isomto \prod_{i=1}^r M_{k_i}(R_i)$, where $M_{k_i}(R_i)$ is the ring of all $k_i$-by-$k_i$ matrices with entries in the ring $R_i:=\End(C_i)$. For any simple abelian variety $C$, the ring $R:=\End(C)$ is a finite integral extension of $\Z$. Therefore we have the following fact. 
\begin{fact}
\label{minimal poly}
Let $A$ be an abelian variety defined over a field of characteristic $0$. For each algebraic group  endomorphism $\phi:A\lra A$ there exists a minimal monic polynomial $f\in \Z[t]$ of degree at most $2\dim(A)$ such that $f(\phi)=0$. 
\end{fact}

The following result is proven in \cite[Corollary~1.2]{Milne}.
\begin{fact}[Rigidity Theorem]
\label{rigidity}
Each endomorphism $\psi:A\lra A$ is of the form $T_y \circ \phi$ for some
$y \in A$, where $T_y:A \to A$ is the translation map $x \mapsto x + y$,
and $\phi\in \End(A)$ is an \emph{algebraic group} endomorphism. In particular, if $\psi$ is a 
finite map, then $\phi :A\lra A$ is an isogeny. Furthermore,  the pair $(T_y,\phi)$ is uniquely 
determined by $\psi$. 
\end{fact}

As a simple consequence of Fact~\ref{rigidity}, we obtain.

\begin{lemma}
\label{commuting endomorphisms}
Let $\psi_1,\psi_2:A\lra A$ be endomorphisms of the form $\psi_i:=T_{y_i}\circ \varphi_i$ 
(for $i=1,2$) where $\varphi_i:A\lra A$ are group endomorphisms. If 
$\psi_1\circ \psi_2 = \psi_2\circ \psi_1$, then 
$\varphi_1\circ \varphi_2 = \varphi_2\circ \varphi_1$.
\end{lemma}

The following result is an immediate application of the structure theorem for the ring of group endomorphisms of an abelian variety.
\begin{fact}
\label{inside a finitely generated group}
Let $S$ be a finitely generated commutative monoid of endomorphisms of an abelian variety $A$ as an algebraic variety. Then for each point $x\in A$, there exists a finitely generated subgroup $\Gamma\subset A$ containing $\OO_S(x)$.
\end{fact} 

\begin{proof}
Let $\{\sigma_1,\dots, \sigma_s\}$ be a set of generators for $S$. For each $i=1,\dots, s$, we let $\gamma_i:=T_{y_i}\circ \tau_i$ for some translations $T_{y_i}$ (where $y_i\in A$) and some group endomorphisms $\tau_i$. Let $d:=\dim(A)$. Then, by Fact~\ref{minimal poly}, for each $i=1,\dots, s$, there exist integers $c_{i,j}$ such that
$$\tau_i^{2d}+c_{i,2d-1}\tau_i^{2d-1}+\cdots + c_{i,1}\tau_i +c_{i,0}\cdot \id=0,$$
where $\id$ always represents the identity map.  
Then $\OO_S(x)$ is contained in the subgroup $\Gamma\subset A$ generated by $\gamma(x), \gamma(y_1),\dots, \gamma(y_s)$, where $\gamma$ varies among the finitely many elements of $S$ of the form $\gamma:=\gamma_1^{m_1}\circ \cdots \gamma_s^{m_s}$, 
with $0\le m_i<2d$, for each $i=1,\dots, s$.
\end{proof}

The next result is a relatively simple application of Fact~\ref{Poincare}.
\begin{lemma}
\label{subabelian}
Let $B\subseteq A$ be an algebraic subgroup of the abelian variety $A$. Then $B\ne A$ if and 
only there exists a nonzero algebraic group endomorphism $\psi:A\lra A$ such that 
$\psi(B)=\{0\}$.
\end{lemma}

\begin{proof}
Clearly, if $B=A$, then there exists no nonzero endomorphism $\psi$ of $A$ such that 
$\psi(B)=\{0\}$. Now, assume $B\ne A$. We note that it suffices to prove the existence of 
$\psi\in\End(A)$ such that $B^0\subseteq \ker(\psi)$, where $B^0$ is the connected component of 
$B$ containing $0$ for if $B^0 \subseteq \ker(\psi)$ and $N := [B:B^0]$ is the index of 
$B^0$ in $B$, then $B \subseteq \ker(\phi)$ where $\phi = [N] \cdot \psi$.  
So, from now on assume $B$ is an abelian subvariety of $A$. We consider $\pi:A\lra A/B$ be the 
canonical quotient map. By Fact~\ref{Poincare}, we obtain that there exists an abelian 
subvariety $C\subseteq A$ and an isogeny $\tau:A/B\lra C$. So, letting $\iota:C\lra A$ be the 
canonical injection map, we get that $\psi:=\iota\circ \tau\circ \pi:A\lra A$ is an endomorphism 
with the property that  $\psi(B)=\{0\}$. We claim that $\psi\ne 0$. Indeed, by construction, the 
image of $\psi$ is $C$ which is a positive dimensional variety (since $B\ne A$).
\end{proof}

The following result is the famous consequence of Mordell-Lang Conjecture proven by 
Faltings~\cite{Faltings}.
\begin{fact}[Faltings' theorem; Mordell-Lang Conjecture]
\label{Faltings theorem}
Let $V\subset A$ be an irreducible subvariety with the property that there exists a finitely 
generated subgroup $\Gamma\subseteq A(\Kbar)$ such that $V(\Kbar)\cap\Gamma$ is Zariski dense in 
$V$. Then $V$ is a coset of an abelian subvariety of $A$.
\end{fact}

We will also employ the following easy result.
\begin{lemma}
\label{one multiple is enough}
Let $A$ be an abelian variety. If $x\in A$ is a point generating a cyclic group which is Zariski dense in $A$, then for each positive integer $\ell$, the cyclic group generated by $\ell x$ is Zariski dense in $A$. 
\end{lemma}

\begin{proof}
Let $H$ be the Zariski closure of the cyclic group generated by $\ell x$; then $H$ is an algebraic subgroup of $A$. Furthermore, because the cyclic group generated by $x$ is Zariski dense in $A$, then 
$$A=\bigcup_{i=0}^{\ell-1} \left(ix + H\right).$$
Since $A$ is connected, we conclude that $H=A$, as desired.
\end{proof}

Finally, for any simple abelian variety $A$ defined over a field $K$ of characteristic $0$, the action of $\Gal(\Kbar/K)$ on $A_{\tor}$ yields the following result.
\begin{fact}
\label{torsion extensions}
The group $\Gal(K(A_{\tor})/K)$ embeds into $\GL_{2d}(\hat{\Z})$, where $d=\dim(A)$ and $\hat{\Z}$ is 
the ring of finite ad\'eles.
\end{fact}


\section{Reductions}
\label{reductions}

Next we proceed with several  preliminary results used later in the proof of 
Theorem~\ref{abelian varieties}. The following result was proven in the case of a cyclic group 
$S$ of automorphisms in \cite{BRS}; we thank Jason Bell for pointing out how to extend the 
result from \cite{BRS} to our setting.

\begin{lemma}
\label{a power suffices 2}
It suffices to prove Theorem~\ref{general theorem} for a submonoid of $S$ spanned by iterates of 
each of the generators of $S$.
\end{lemma}

\begin{proof}
We consider a finite generating set $\cU:=\{\gamma_1,\dots, \gamma_s\}$ for the monoid $S$. We 
assume $S$ does not fix a non-constant fibration of $A$ (otherwise Theorem~\ref{general theorem} 
holds).   
We let $S'$ be the submonoid of $S$ spanned by the endomorphisms in 
$\cU' :=\{\gamma_1^{m_1},\dots, \gamma_s^{m_s}\}$ (for some positive integers $m_i$). We assume 
Theorem~\ref{general theorem} holds for $S'$. If also $S'$ does not fix a non-constant 
fibration, then there exists $x\in A(\Kbar)$ such that the $S'$-orbit of $x$ is Zariski dense in 
$A$; hence also $\OO_S(x)$ is Zariski dense in $A$. So, it remains to prove that $S'$ cannot fix 
a non-constant fibration if $S$ does not fix a non-constant fibration.

We assume $f\circ \gamma_i^{m_i}=f$ for some non-constant map $f:A\lra \bP^1$ (for each $i$). 
Let $\cS$ be a finite set of representatives for the cosets of $S'$ in $S$ (note that $S/S'$ is 
a finite group since it is a finite monoid in which each element is invertible); without loss of 
generality we assume the identity is part of $\cS$. Let $m:=|\cS|$, and let 
$\cS:=\{\sigma_1,\dots,\sigma_m\}$. 
Let $s_1, \ldots, s_m$ be the elementary symmetric functions $g_i:(\bP^1)^m \to \bP^1$
and let $g_i := s_i(f \circ \sigma_1, \ldots, f \circ \sigma_m)$ 
(for $i=1,\dots, m$)  Clearly, $\gamma_i$ preserves each 
fibration $g_j$; hence if one $g_j$ is non-constant, then we are done. If each $g_j$ is a 
constant, then we obtain a contradiction because $f=f\circ \id$ would be a root of the 
polynomial (with constant coefficients)
$$X^m-g_1X^{m-1}+g_2X^{m-2}+\cdots +(-1)^mg_m=0.$$
This completes the proof of Lemma~\ref{existence of a point 2}.
\end{proof}

\begin{lemma}
\label{T and barT}
With the notation as in Theorem~\ref{general theorem}, let $T$ be a submonoid of $S$ such that $\bar{T}=S$. If the conclusion of Theorem~\ref{general theorem} holds for $T$, then it holds for $S$.
\end{lemma}

\begin{proof}
We assume that there exists no non-constant fibration preserved by all elements of $S$, and it 
suffices to prove that there is also no non-constant fibration preserved by the elements of $T$. 
Assume, by contradiction that there exists $f:A\lra \bP^1$ such that $f\circ \psi=f$ for each 
$\psi\in T$. Now, let $\sigma\in S$; then there exist $\psi_1,\psi_2\in T$ such that 
$\gamma\psi_1=\psi_2$. So, using also that $S$ is commutative, we get
$$f\circ \gamma = f\circ \psi_1\circ \gamma = f\circ \psi_2 = f.$$
Hence $f$ must be constant, as desired.
\end{proof}

Combining Lemmas~\ref{a power suffices 2} and \ref{T and barT} we obtain the following reduction 
of Theorem~\ref{general theorem}.
\begin{lemma}
\label{a power suffices 3}
With the notation from Theorem~\ref{general theorem}, assume the monoid $S$ is generated by the 
maps $\gamma_1,\dots, \gamma_s$. Then it suffices to prove the conclusion of 
Theorem~\ref{general theorem} for a finitely generated submonoid $T$ of $S$ with the property 
that $\gamma_i^n\in \bar{T}$ for each $i=1,\dots, s$, for some positive integer $n$.
\end{lemma}


\section{Auxiliary results}
\label{necessary lemmas}

In this Section we present several technical results useful for our proof of Theorems~\ref{abelian varieties} and \ref{general theorem}. 

\begin{lemma}
\label{existence of a point 2}
Let $K$ be a finitely generated field of characteristic $0$, and let $\Kbar$ be an algebraic 
closure of $K$. 
Let $\psi_1, \dots, \psi_s:B\lra C$ be algebraic group morphisms of  abelian varieties defined 
over $K$, and let $y_1,\dots, y_s\in C(K)$. Then there exists $x\in B(\Kbar)$ such that for each 
$i=1,\dots, s$, the Zariski closure of the subgroup generated by $\psi_i(x)+y_i$ is the 
algebraic group generated by $\psi_i(B)$ and $y_i$.
\end{lemma} 

\begin{proof}
We let $B=A_1+\cdots + A_m$ written as a sum of simple abelian varieties. 

Let $i=1,\dots, s$; then $\psi_i(B)$ equals the sum $\psi_i(A_1)+\cdots + \psi_i(A_m)$ (each 
algebraic group being either simple or trivial). We find an algebraic point $x_i\in \psi_i(B)$ 
such that the Zariski closure of the cyclic group generated by $x_i+y_i$ is the algebraic group 
generated by $\psi_i(B)$ and $y_i$; moreover we ensure that 
$$\cap_{i=1}^s \psi_i^{-1}(\{x_i\})$$ is nonempty (in $B$). We find $x_i$ as  a sum 
$x_{i,1}+\cdots + x_{i,m}$, where each $x_{i,j}\in \psi_i(A_j)$. If for some $j$ we have 
$\psi_i(A_j)=\{0\}$, we simply pick $x_{i,j}=0$. Then our goal is to construct the sequence 
$\{x_{i,j}\}$ such that for each $j=1,\dots, m$, the set 
\begin{equation}
\label{nonempty intersection}
\cap_{i=1}^s (\psi_i)|_{A_j}^{-1}\left(\{x_{i,j}\}\right)
\end{equation}
is nonempty (in $A_j$). Obviously when $\psi_i(A_j)=\{0\}$, we might as well disregard the set 
$$(\psi_i)|_{A_j}^{-1}(\{x_{i,j}\})=(\psi_i)|_{A_j}^{-1}(\{0\})=A_j$$
from the above intersection. Let now $j=1,\dots, m$ such that $\psi_i(A_j)$ is nontrivial. We will 
show that there exists $x_{i,j}\in \psi_i(A_j)$ such that for any positive integer $n$ we have
\begin{equation}
\label{necessary technical condition}
nx_{i,j}\notin \left(\psi_i(A_j)\right)\left(K\left(C_{\tor},  x_{i,1},\dots, 
x_{i,j-1}\right)\right).
\end{equation}

\begin{claim}
\label{one technical claim}
If the above condition \eqref{necessary technical condition} holds for each $j=1,\dots, m$ such that $\psi_i(A_j)\ne \{0\}$, then the Zariski closure of the cyclic group generated by 
$x_i+y_i$ is the algebraic subgroup $B_i$ generated by $\psi_i(B)$ and $y_i$.
\end{claim}

\begin{proof}[Proof of Claim~\ref{one technical claim}.]
Indeed, assume 
there exists some algebraic subgroup $D\subseteq C$ (not necessarily connected) such that 
$x_i+y_i\in D(\Kbar)$. Let $j\le m$ be the largest integer such that $x_{i,j}\ne 0$; then we have
$$x_{i,j}\in \left((-y_i-x_{i,1}-\cdots -x_{i,j-1})+D\right)\cap \psi_i(A_j).$$
Assume first that $\psi_i(A_j)\cap D$ is a proper algebraic subgroup of $\psi_i(A_j)$. Since 
$\psi_i(A_j)$ is a simple abelian variety, then $D\cap \psi_i(A_j)$ is a $0$-dimensional 
algebraic subgroup of $C$; hence there exists a nonzero integer $n$ such that 
$n\cdot (D\cap \psi_i(A_j))=\{0\}$. Then $nx_{i,j}$ is the only (geometric) point of the 
subvariety $n\cdot \left(\left((-y_i-x_{i,1}-\cdots -x_{i,j-1})+D\right)\cap \psi_i(A_j)\right)$ 
which is thus rational over $K(C_{\tor},x_{i,1},\dots, x_{i,j-1})$. But by our 
construction, $$nx_{i,j}\notin \psi_i(A_j)(K(C_{\tor}, x_{i,1},\dots, x_{i,j-1}))$$ 
which is a contradiction. Therefore $\psi_i(A_j)\subseteq D$ if $j$ is the largest index $\le m$ such that $x_{i,j}\ne 0$ (or equivalently, such that $\psi_i(A_j)\ne 0$).   So, $x_i+y_i\in D$ yields now $x_i'+y_i\in D$, where 
$x_i':=x_{i,1}+\cdots +x_{i,j-1}$. Repeating the exact same argument as above for the next positive integer $j_1<j$ for which $\psi_i(A_{j_1})\ne \{0\}$, and then arguing inductively we obtain that 
each $\psi_i(A_j)$ is contained in $D$, and therefore $\psi_i(B)\subseteq D$. But then 
$x_i\in \psi_i(B)\subseteq D$ and so, $y_i\in D$ as well, which yields that  the Zariski 
closure of the group generated by $x_i+y_i$ is the algebraic subgroup $B_i$ of $C$ generated by $\psi_i(B)$ 
and $y_i$.  
\end{proof}
 
We just have to show that we can choose $x_{i,j}$ both satisfying \eqref{necessary technical 
condition} and also such that the above intersection \eqref{nonempty intersection} is nonempty. 
So, the problem reduces to the following: $L$ is a finitely generated field of characteristic 
$0$, $\varphi_1,\dots,  \varphi_\ell$ are algebraic group homomorphisms (of finite kernel)  between a simple abelian variety $A$ and another abelian variety  $C$ all defined over $L$, and we want to find 
$z\in A(\Kbar)$ such that for each positive integer $n$, and for each $i=1,\dots, \ell$, we have
\begin{equation}
\label{necessary technical condition 2}
n\varphi_i(z)\notin \varphi_i(A)\left(L\left(C_{\tor}\right)\right).
\end{equation}
Indeed, with the above notation, $A:=A_j$, $L$ is the extension of $K$ generated  by $x_{i,k}$ 
(for $i=1,\dots, s$, and $k=1,\dots, j-1$), and the $\varphi_i$'s are the homomorphisms 
$\psi_i$'s (restricted on $A=A_j$) for which $\psi_i(A_j)$ is nontrivial. 

Let $d$ be the maximum of the degree of the isogenies $\varphi_i':A\lra \varphi_i(A)\subset C$. In particular, this means that for each $w\in C(\Kbar)$, and for each $z\in A(\Kbar)$ such that $\phi_i(z)=w$ we have
\begin{equation}
\label{key index inequality}
\left[L(z):L\right]\le d\cdot \left[L(w):L\right].  
\end{equation}
For any subfield $M\subseteq \Kbar$, we let 
$M^{(d)}$ be the compositum of all extensions of $M$ of degree at most equal to $d$. 

\begin{claim}
\label{simple groups claim}
Let $L$ be a finitely generated field of characteristic $0$, let $C$ be an abelian variety defined over $L$, let $L_{\tor}:=L(C_{\tor})$, and let $d$ be a positive integer. Then there exists a normal extension of $L_{\tor}^{(d)}$ whose Galois group is not abelian.
\end{claim}

\begin{proof}[Proof of Claim~\ref{simple groups claim}.]
As proven in \cite{Thornhill}, the field $L_{\tor}$ is Hilbertian (note that $L$ itself is Hilbertian since it is a finitely generated field of characteristic $0$).  For each positive integer $n$, according to \cite[Corollary~16.2.7 (a)]{Fried}, there exists a Galois extension $L_n$ of $L_{\tor}$ such that $\Gal(L_n/L_{\tor})\isomto S_n$ (the symmetric group on $n$ letters).  Assume there exists a abelian extension $L_0$ of $L_{\tor}^{(d)}$ containing $L_n$. If $n>\max\{5,d!\}$, we will derive a contradiction from our assumption. 

We let $G_1:=\Gal\left(L_0/L_{\tor}^{(d)}\right)$ and $G_0:=\Gal(L_0/L_{\tor})$. Then there exists a surjective group homomorphism $f:G_0\lra S_n$. Because $G_1$ is a normal subgroup of $G_0$ (and $f$ is a surjective group homomorphism), we get that $f(G_1)$ is a normal subgroup of $S_n$, and moreover, it is abelian since $G_1$ is abelian. Because $n\ge 5$, the only proper normal subgroup of $S_n$ is $A_n$, which is not abelian. Hence, $G_1\subseteq \ker(f)$, and therefore, $f$ induces a surjective group homomorphism (also denoted by $f$) from $G_0/G_1$ to $S_n$; more precisely, we have a surjective group homomorphism $f:G^{(d)}\lra S_n$, where $G^{(d)}:=\Gal\left(L_{\tor}^{(d)}/L_{\tor}\right)$. 
But $G^{(d)}$ is a group of exponent $d!$, and so,  $S_n=f(G^{(d)})$ is also a group of exponent $d!$, which is a contradiction with the fact that $n>d!$.    
\end{proof}
Claim~\ref{simple groups claim} yields that there exists a point $z\in A(\Kbar)$ which is not defined over a abelian extension of $L\left(C_{\tor}\right)^{(d)}$; i.e., $nz\notin A\left(L\left(C_{\tor}\right)^{(d)}\right)$ for all positive integers $n$. Hence, $n\phi_i(z)\notin \phi_i(A)\left(L\left(C_{\tor}\right)\right)$ (see \eqref{key index inequality}), which  concludes the proof of Lemma~\ref{existence of a point 2}.
\end{proof}

The next result will be used (only) in the proof of Theorem~\ref{abelian varieties}.
\begin{lemma}
\label{replacing by a conjugate}
It suffices to prove Theorem~\ref{abelian varieties} for a conjugate 
$\gamma^{-1}\circ \sigma\circ \gamma$ of the automorphism $\sigma$ under some automorphism 
$\gamma$.
\end{lemma}

\begin{proof}
Since $\OO_{\gamma^{-1}\sigma\gamma}(\gamma^{-1}(x))=\gamma^{-1}\left(\OO_\sigma(x)\right)$, we 
obtain that there exists a Zariski dense orbit of an algebraic point under the action of 
$\sigma$ if and only if there exists a Zariski dense orbit of an algebraic point under the 
action of $\gamma^{-1}\circ \sigma\circ\gamma$. Also, $\sigma$ preserves a non-constant fibration $f:A\lra \bP^1$ if and only if 
$\gamma^{-1}\sigma\gamma$ preserves the non-constant fibration $f\circ \gamma$.
\end{proof}

The conclusion of the next result shares the same philosophy as the conclusion of Lemma~\ref{existence of a point 2}: one can find an algebraic point in an abelian variety so that it is \emph{sufficiently} generic with respect to any given set of finitely many points.

\begin{lemma}
\label{isogenies and m}
Let $\Gamma\subseteq A(K)$ be a subgroup such that $\End(A)\otimes_\Z\Gamma$ is a finitely 
generated $\End(A)$-module, and let $B\subseteq A$ be a nontrivial abelian subvariety. Then 
there exists $x\in B(\Kbar)$ such that for each $\psi\in\End(A)$ satisfying $\psi(x)\in\Gamma$, we must 
have that $B\subseteq \ker(\psi)$.
\end{lemma}

\begin{proof}
Each abelian variety is isogenous to a product of simple abelian varieties; so let 
$\pi:A\lra A_0:=\prod_{i=1}^r C_i^{k_i}$ be such an isogeny, where each $C_i$ is a simple 
abelian variety. Then it suffices to find an algebraic point $y\in C:=\pi(B)$ such that for each 
$\phi\in\End(A_0)$, if $\phi(y)\in \pi(\Gamma)$, then $C\subseteq\ker(\phi)$.

At the expense of replacing $C$ with an isogenous abelian variety, we may assume that  
$C:=\prod_{i=1}^r C_i^{m_i}$ with $0\le m_i\le k_i$. Each endomorphism $\phi\in \End(A_0)$ is of 
the form $(J_1,\dots, J_r)$ where each $J_i\in M_{k_i}(R_i)$, where $M_{k_i}(R_i)$ is the $k_i$-by-$k_i$ matrices with entries in the ring $R_i$ of endomorphisms of $C_i$ (note that $R_i$ is a 
finite integral extension of $\Z$). We let $\Gamma_i$ be the finitely generated $R_i$-module 
generated by the projections of $\pi(\Gamma)$ on each of the $k_i$ copies of $C_i$ contained in 
the presentation of $A_0=\prod_{i=1}^r C_i^{k_i}$. We let $y_{i,1},\dots, y_{i,\ell_i}$ be   
generators of the free part of $\Gamma_i$ as an $R_i$-module. Without loss of generality, we may assume the points $y_{i,1},\dots, y_{i,\ell_i}$ are linearly independent over $R_i$. 

Then it suffices to pick $x\in C$ of the form 
$$(x_{1,1},\dots, x_{1,m_1},x_{2,1},\dots, x_{2,m_2},\dots, x_{r,1},\dots, x_{r,m_r}),$$
where each $x_{i,j}\in C_i$ such that for each $i$, the points 
$x_{i,1},\dots, x_{i,m_i}, y_{i,1},\dots, y_{i,\ell_i}$ are linearly independent over $R_i$. The 
existence of such points $x_{i,j}$ follows from the fact that each 
$C_i(\Kbar)\otimes_{R_i}\Frac(R_i)$ has the structure of a $\Frac(R_i)$-vector space of infinite 
dimension. 
\end{proof}

The next result is an application of Fact~\ref{Faltings theorem}.

\begin{lemma}
\label{combinatorial lemma}
Let  $y_1,\dots, y_r\in A(K)$, and let $P_1,\cdots, P_r\in \Q[z]$ such that $P_i(n)\in\Z$ for 
each $n\ge 1$ and for each $i=1,\dots, r$, while $\deg(P_r)>\cdots >\deg(P_1)>0$. Then for each 
infinite subset $S\subseteq \N$, there exist nonzero integers $\ell_1,\dots, \ell_r$ such that 
the Zariski closure $V$ of  the set 
$$\{P_1(n)y_1+\cdots +P_r(n)y_r\colon n\in S\}$$  
contains  a coset of the subgroup $\Gamma$ generated by $\ell_1y_1,\cdots, \ell_ry_r$.  
\end{lemma}

\begin{proof}
Because $V(K)\cap\Gamma$ is Zariski dense in $V$, then by Fact~\ref{Faltings theorem} we obtain that $V$ is a finite union of cosets of algebraic subgroups of $A$.  So, at the expense of replacing $S$ by an infinite subset, we may assume $V=z+C$, for some $z\in A(K)$ and some irreducible algebraic subgroup $C$ of $A$. This is equivalent with the existence of an endomorphism $\psi:A\lra A$ such that $\ker(\psi)^0=C$ (the construction of $\psi$ is identical with the one given in the proof of Lemma~\ref{subabelian}); hence $\psi$ is constant on the set $\{P_1(n)y_1+\cdots P_r(n)y_r\}_{n\in S}$. We will show there exist nonzero integers $\ell_i$ such that $\ell_iy_i\in \ker(\psi)$ for each $i=1,\dots, r$; since $\ker(\psi)^0=C$, then we obtain the desired conclusion.

We proceed by induction on $r$. The case $r=1$ is obvious since then $\{P_1(n)\}_{n\in S}$ takes infinitely many distinct integer values (note that $\deg(P_1)\ge 1$), and so, if $\psi$ is constant on the set $\{P_1(n)y_1\}_{n\in S}$, then $\psi(\ell y_1)=0$ for some nonzero $\ell:=P_1(n)-P_1(n_0)$ with distinct $n_0,n\in S$.  Next we assume the statement holds for all $r<s$ (where $s\ge 2$), and we prove it for $r=s$.  

Let $n_0\in S$. At the expense of replacing each $P_i(n)$ by $P_i(n)-P_i(n_0)$, we may assume from now on that the set $\{P_1(n)y_1+\cdots P_s(n)y_s\}_{n\in S}$ lies in the kernel of $\psi$. Let $n_1\in S$ such that $P_1(n_1)\ne 0$ (note that $\deg(P_1)\ge 1$), and for each $i=2,\dots, s$ we let $Q_i(z):=P_1(n_1)\cdot P_i(n)-P_1(n)\cdot P_i(n_1)$. Then the set $\{\sum_{i=2}^s Q_i(n)y_i\}_{n\in S}$ is in the kernel of $\psi$. Because $\deg(Q_i)=\deg(P_i)$ for each $i=2,\dots, s$, we can use the induction hypothesis and conclude that there exist nonzero integers $\ell_2,\dots, \ell_s$ such that $\ell_iy_i\in\ker(\psi)$ for each $i$. Since $\psi(P_1(n_1)y_1+\cdots +P_s(n_1)y_s)=0$ and $P_1(n_1)\ne 0$, then also $\left(P_1(n_1)\cdot\prod_{i=2}^s\ell_i \right)y_1\in\ker(\psi)$. This concludes our proof.
\end{proof}

The above Lemma has the following important consequence for us.
\begin{lemma}
\label{infinitely many iterates nilpotent operator}
Let $K$ be a field of characteristic $0$, let $\Kbar$ be an algebraic closure of $K$, let $A$ be an abelian variety defined over $K$, let $\tau \in \End(A)$ with the property that there exists a positive integer $r$ such that $(\tau-\id)^r=0$, let $y\in A(\Kbar)$, let $\sigma:A\lra A$ be an endomorphism as algebraic varieties such that $\sigma=T_y\circ \tau$, and let $x\in A(\Kbar)$. Let $\gamma\in\End(A)$ with the property that there exists an infinite set $S$ of positive integers such that $\gamma$ is constant on the set $\{\sigma^n(x)\colon n\in S\}$. Then there exists a positive integer $\ell$ such that $\ell\cdot\left(\beta(x)+ y\right)\in \ker(\gamma)$, where $\beta:=\tau-\id$. 
\end{lemma}

\begin{proof}
We compute $\sigma^n(x)$ for any $n\in\N$; first of all, we have
\begin{equation}
\label{formula of an iterate}
\sigma^n(x)=\tau^n(x)+\sum_{i=0}^{n-1}\tau^i(y).
\end{equation}
Then (since $\beta=\tau-\id$ and also) noting that $\beta^r=0$ we have
\begin{align}
& \sigma^n(x)\\
& =\sum_{i=0}^n\binom{n}{i}\beta^i(x) + \sum_{i=0}^{n-1}\tau^i(y)\\
& = \sum_{i=0}^{r-1}\binom{n}{i}\beta^i(x) + \sum_{i=0}^{n-1}\sum_{j=0}^i \binom{i}{j}\beta^j(y)\\
& = \sum_{j=0}^{r-1}\binom{n}{j}\beta^j(x) + \sum_{j=0}^{r-1} \left(\sum_{i=j}^{n-1}\binom{i}{j}\right)\beta^j(y)\\
& = \sum_{j=0}^{r-1}\binom{n}{j}\beta^j(x) + \sum_{j=0}^{r-1} \binom{n}{j+1}\beta^j(y)\\
& =  x+\sum_{j=1}^{r} \binom{n}{j}\beta^{j}(x) + \sum_{j=1}^r \binom{n}{j}\beta^{j-1}(y)\\
\label{binomial formula}
& = x+\sum_{j=1}^r \binom{n}{j}\beta^{j-1}\left(\beta(x)+y\right).
\end{align}
Since $\gamma$ is constant on the set $\{\sigma^n(x)\colon n\in S\}$, then letting $n_1\in S$ we have that for each $n\in S$, 
\begin{equation}
\label{many points in the kernel}
\sum_{j=1}^r \left(\binom{n}{j}-\binom{n_1}{j}\right)\beta^{j-1}\left(\beta(x)+y\right)\in \ker (\gamma).
\end{equation}
Using Lemma~\ref{combinatorial lemma} and \eqref{many points in the kernel}, we obtain the desired conclusion.
\end{proof}

Then the following result is an immediate consequence of Lemma~\ref{infinitely many iterates nilpotent operator} and of Lemma~\ref{one multiple is enough}.
\begin{cor}
\label{immediate corollary to many iterates}
With the notation as in Lemma~\ref{infinitely many iterates nilpotent operator}, if the cyclic group generated by $\beta(x)+y$ is Zariski dense in $A$, then $\gamma=0$. Moreover, the set $\{\sigma^n(x)\colon n\in S\}$ is Zariski dense in $A$.
\end{cor}

\begin{proof}
Indeed, Lemmas~\ref{one multiple is enough} and \ref{infinitely many iterates nilpotent operator} yield that any group homomorphism $\gamma$ which is constant on the set $U:=\{\sigma^n(x)\colon n\in S\}$ must be trivial.  

Now, for the `moreover' part of Corollary~\ref{immediate corollary to many iterates}, Fact~\ref{inside a finitely generated group} yields that $U$ (along with $\OO_\sigma(x)$) is contained in a finitely generated subgroup of $A$, and so, Fact~\ref{Faltings theorem} yields that the Zariski closure of $U$ is a finite union of cosets of algebraic subgroups of $A$. Pick such a coset $w+H$ which contains infinitely many $\sigma^n(x)$. Then another application of Lemma~\ref{infinitely many iterates nilpotent operator} (coupled with Lemmas~\ref{subabelian} and \ref{one multiple is enough}) yields that $H=A$, thus completing our proof that $U$ is Zariski dense in $A$.   
\end{proof}


\section{The cyclic case}
\label{one section}

Now we are ready to prove Theorem~\ref{general theorem} for cyclic monoids.
\begin{proof}[Proof of Theorem~\ref{abelian varieties}.]
By Fact~\ref{rigidity}, there exists an isogeny $\tau :A\lra A$, and there exists  $y\in A(K)$, such that  $\sigma(x)=\tau(x)+y$ for all $x\in A$.  
At the expense of replacing $\sigma$ by an iterate $\sigma^n$ (and in particular, replacing $\tau$ by $\tau^n$; see also \eqref{formula of an iterate}), we may assume $\dim\ker(\tau^m-\id)=\dim(\ker(\tau-\id))$ for all $m\in\N$ (see  Lemma~\ref{a power suffices 2} which shows that it is sufficient to prove Theorem~\ref{abelian varieties} for an iterate of $\sigma$). In other words, we may assume that the only root of unity, if any, which is a root of the minimal polynomial $f$ (with coefficients in $\Z$) of $\tau\in\End(A)$ is equal to $1$.

Let $r$ be  the order of vanishing at $1$ of $f$, and let $f_1\in \Z[t]$ such that $f(t)=f_1(t)\cdot (t-1)^r$. Then $f_1$ is also a monic polynomial, and if $r=0$, then $f_1=f$. Let $A_1:=(\tau-\id)^r(A)$ and let $A_2:=f_1(\tau)(A)$, where $f_1(\tau)\in\End(A)$ and $\id$ is the identity map on $A$. If $r=0$, then $A_2=0$ and therefore $A_1=A$. By definition, both $A_1$ and $A_2$ are connected algebraic subgroups of $A$, hence they are both abelian subvarieties of $A$. Furthermore, by definition, the restriction of $\tau|_{A_1}\in\End(A_1)$ has  minimal polynomial equal to $f_1$ whose roots are not roots of unity. On the other hand, $(\tau-\id)^r|_{A_2}=0$. 

\begin{lemma}
\label{decomposition of A as a sum}
With the above notation, $A=A_1+A_2$.
\end{lemma}

\begin{proof}[Proof of Lemma~\ref{decomposition of A as a sum}.]
By the definition of $r$ and of $f_1$, we know that the polynomials $f_1(t)$ and $(t-1)^r$ are coprime; so there exist polynomials $g_1,g_2\in\Z[t]$ and there exists a nonzero integer $k$ (the resultant of $f_1(t)$ and of $(t-1)^r$) such that
$$f_1(t)\cdot g_1(t)+(t-1)^r\cdot g_2(t)=k.$$
Let $x\in A(\Kbar)$ and let $x_0\in A(\Kbar)$ such that $kx_0=x$. Then clearly $$x_1:=(\tau-\id)^r\left(g_2(\tau)x_0\right)\in A_1\text{ and }x_2:=f_1(\tau)\left(g_1(\tau)x_0\right)\in A_2,$$ 
and moreover, $x_1+x_2=kx_0=x$, as desired. 
\end{proof}

Arguing similarly, one can show that $A_1\cap A_2\subseteq A[k]$ since if $x\in A_1\cap A_2$ then $f_1(\tau)x=0=(\tau-\id)^rx$ and thus $$kx=\left(g_1(\tau)f_1(\tau)+g_2(\tau)(\tau-\id)^r\right)x=0.$$ 

Let $y_1\in A_1$ and $y_2\in A_2$ such that $y=y_1+y_2$; furthermore, we may assume that if $y_1\in A_2$ then $y_1=0$. We note that $\tau$ restricts to an endomorphism to each $A_1$ and $A_2$; we denote by $\tau_i$ the action of $\tau$ on each $A_i$. Let $y_0\in A_1(\Kbar)$ such that $(\id-\tau_1)(y_0)=y_1$ (note that $(\id-\tau_1):A_1\lra A_1$ is an isogeny because the minimal polynomial $f_1$ of $\tau_1\in\End(A_1)$ does not have the root $1$). Using Lemma~\ref{replacing by a conjugate}, it suffices to prove Theorem~\ref{abelian varieties} for $T_{-y_0}\circ \sigma\circ T_{y_0}$; so, we may and do assume that $y_1=0$.  

Let $\sigma_i:A_i\lra A_i$ be given by $\sigma_1(x)=\tau_1(x)$ and $\sigma_2(x)=\tau_2(x)+y_2$. Then for each $x\in A$, we let $x_1\in A_1$ and $x_2\in A_2$ such that $x=x_1+x_2$; we have:
$$\sigma(x)=\sigma(x_1+x_2)=\tau(x_1+x_2)+y_2=\tau(x_1)+\tau(x_2)+y_2=\sigma_1(x_1)+\sigma_2(x_2).$$
Moreover, $\sigma^n(x_1+x_2)=\sigma_1^n(x_1)+\sigma_2^n(x_2)$ for all $n\in\N$.

We let $\beta:=(\tau_2-\id)|_{A_2}\in\End(A_2)$; then $\beta^r=0$. Let $B$ be the Zariski closure of the subgroup of $A_2$ generated by $\beta(A_2)$ and $y_2$. Then $B$ is an algebraic subgroup of $A_2$.

\begin{lemma}
\label{fibration exists}
If $B\ne A_2$, then $\sigma$ preserves a nonconstant fibration.
\end{lemma} 

\begin{proof}[Proof of Lemma~\ref{fibration exists}.]
If $B\ne A_2$, then $\dim(B)<\dim(A_2)$ (note that $A_2$ is connected) and since $A_2\cap A_1$ is finite, we conclude that the algebraic subgroup $C:=A_1+B$ is a proper abelian subvariety of $A$. We let $f:A\lra A/C$ be the quotient map; we claim that $f\circ \sigma = f$. Indeed, for each $x\in A$, we let $x_1\in A_1$ and $x_2\in A_2$ such that $x=x_1+x_2$ and then
$$f(\sigma(x))=f(\sigma(x_1+x_2))=f(\sigma_1(x_1)+\sigma_2(x_2))=f(\sigma_2(x_2))=f(x_2)=f(x).$$
Since $A/C$ is a positive dimensional algebraic group and $f:A\lra A/C$ is the quotient map, then we conclude that $\sigma$ preserves a nonconstant fibration.
\end{proof}

From now on, assume $B=A_2$. We will prove that there exists $x\in A(\Kbar)$ such that $\OO_\sigma(x)$ is Zariski dense in $A$. First we prove there exists $x_2\in A_2(\Kbar)$ such that $\OO_{\sigma_2}(x_2)$ is Zariski dense in $A_2$. 

Because we assumed that the group generated by $\beta(A_2)$ and $y_2$ is Zariski dense in $A$, then  Lemma~\ref{existence of a point 2} yields the existence of $x_2\in A_2(\Kbar)$ such that the group generated by $\beta(x_2)+y_2$ is Zariski dense in $A_2$. Then Corollary~\ref{immediate corollary to many iterates} yields that any infinite subset of $\OO_{\sigma_2}(x_2)$ is Zariski dense in $A_2$. If $A_1$ is trivial, then $A_2=A$ and $\sigma_2=\sigma$ and Theorem~\ref{abelian varieties} is proven. So, from now on, assume that $A_1$ is positive dimensional.

Let $\Gamma$ be the  subgroup of $A(\Kbar)$ generated by all $\phi(x_2)$ and $\phi(y_2)$ as we vary $\phi\in \End(A)$. Then $\Gamma$ is a finitely generated $\End(A)$-module. Using Lemma~\ref{isogenies and m}, we may find $x_1\in A_1(\Kbar)$ with the property that if $\psi\in\End(A)$ has the property that $\psi(x_1)\in \Gamma$, then $A_1\subseteq \ker(\psi)$. Let $x:=x_1+x_2$; we will prove that $\OO_\sigma(x)$ is Zariski dense in $A$.

Let $V$ be the Zariski closure of $\OO_\sigma(x)$. The orbit $\OO_\sigma(x)$ is contained in a finitely generated group (see Fact~\ref{inside a finitely generated group}). Then Fact~\ref{Faltings theorem} yields that $V$ is a finite union of cosets of algebraic subgroups of $A$. So, if $V\ne A$, then there exists a coset $c+C$ of a proper algebraic subgroup $C\subset A$ which contains  $\{\sigma^n(x)\}_{n\in S}$ for some infinite subset $S\subseteq \N$. By Lemma~\ref{subabelian}, there exists a nonzero $\psi\in\End(A)$ such that $\psi(\sigma^n(x))=\psi(c)$ for each $n\in S$, i.e. $\psi$ is constant on the set $\{\sigma^n(x)\colon n\in S\}$. 

Let $ n>m$ be two elements of $S$. Then $\psi(\sigma^n(x)-\sigma^m(x))=0$, and so,
$$\psi(\tau_1^n-\tau_1^m)(x_1)=\psi(\sigma_2^m-\sigma_2^n)(x_2)\in \Gamma.$$
Using the fact that $x_1\in A_1$ was chosen to satisfy the conclusion of   Lemma~\ref{isogenies and m} with respect to $\Gamma$ and the fact that $\tau_1^n-\tau_1^m=\tau_1^m(\tau_1^{n-m}-\id)$ is an isogeny on $A_1$, we obtain that $\psi(A_1)=0$. Thus $\psi$ is constant on $\{\sigma_2^n(x_2)\}_{n\in S}$.  Then Corollary~\ref{immediate corollary to many iterates} yields that $A_2\subseteq \ker(\psi)$. Hence $A_1+A_2=A\subseteq \ker(\psi)$ which contradicts the fact that $\psi\ne 0$. 
This concludes our proof.
\end{proof}


\section{The general case}
\label{many section}

The proof of Theorem~\ref{general theorem} follows the same strategy as the proof of Theorem~\ref{abelian varieties}.
\begin{proof}[Proof of Theorem~\ref{general theorem}.]
We let  
$\gamma_1,\dots, \gamma_s$ be a set of generators for $S$. We let $S_0$ be the monoid of group 
endomorphisms of $A$ consisting of all $\tau:A\lra A$ such that there exists some $y\in A$ such 
that $T_y\circ \tau\in S$. We let $\cU:=\{\gamma_1,\dots, \gamma_s\}$, and also let $\cU_0$ be a 
finite set of generators for $S_0$ corresponding to the elements in $\cU$ (i.e., for each 
$\varphi\in \cU_0$, there exists $y\in A$ such that $T_y\circ \varphi\in \cU$).

By Fact~\ref{Poincare}, $A$ is isogenuous with a product of simple abelian varieties 
$\prod_i A_i^{r_i}$ and so, $\End(A)$ (the ring of group endomorphisms of $A$) is isomorphic to 
$\prod_i M_{r_i}(\End(A_i))$. We let $R_i:=\End(A_i)$ and $F_i:=\Frac(R_i)$. Then each element 
in $S_0$ is represented by a tuple of matrices in $\prod_i M_{r_i}(R_i)$; from now on, we use freely this 
identification of the group endomorphisms from $S_0$ with tuples of matrices in $\prod_i M_{r_i}(R_i)$. 
Using Lemma~\ref{no root of unity} and also Lemma~\ref{a power suffices 2}, it suffices to 
assume that for each $\tau\in S_0$, and for each positive integer $n$, we have
\begin{equation}
\label{equation no root of unity}
\dim\ker(\tau-\id)=\dim\ker(\tau^n-\id).
\end{equation}
Let $U_0$ be the submonoid of $S_0$ generated by all $\tau\in S_0$ such that 
\begin{equation}
\label{minimal dimension definition}
\max_{n\ge 1}\dim \ker(\tau -\id)^n
\end{equation} 
is minimal as we vary $\tau$ in $S_0$. Then, by Lemma~\ref{minimal eigenspace for 1 linear algebra}, $\bar{U_0}=S_0$. Let $U$ be the submonoid of $S$ corresponding to $U_0$, i.e. the set of all $\sigma\in S$ such that there exists some $\tau\in U_0$ and there exists a translation $T_y$ on $A$ for which $\sigma =  T_y\circ \tau$.   Because $\bar{U_0}=S_0$,  then also $\bar{U}=S$. Using Lemma~\ref{T bar T}, there exists a finitely generated submonoid $U'$ of $U$ (and therefore of $S$) and there exists a positive integer $n$ such that for each $i=1,\dots, s$, we have $\gamma_i^n\in \bar{U'}$. By Lemma~\ref{a power suffices 3}, it suffices to prove Theorem~\ref{general theorem} for $U'$. So, from now on, we assume $U'=S$. In particular, this means that $S_0$ is generated (as a monoid) by finitely many endomorphisms  $\tau$ satisfying \eqref{minimal dimension definition}; we denote this set by $\cU_0$ (as before).  Finally, we recall our notation that  $\cU=\{\gamma_1,\dots,\gamma_s\}$ is a finite set of generators of $ S$, and that for each generator $\tau\in \cU_0$ of $S_0$ there exists some translation $T_y$ and some $i=1,\dots, s$ such that $T_y\circ \tau = \gamma_i$.

Let $\tau_1,\tau_2$ in $\cU_0$. Assume $r_1$ is the order of the root $1$ of the minimal 
polynomial for $\tau_1$, and let $B_2:=\ker(\tau_1-\id)^{r_1}$. Since $\tau_2$ commutes with 
$\tau_1$, we obtain that $\tau_2$ acts on $B_2$. Furthermore, because both $\tau_1$ and $\tau_2$ 
are in $\cU_0$, it must be that the restriction of the action of $\tau_2$ on $B_2$ is also 
unipotent (see also the proof of Lemma~\ref{minimal eigenspace for 1 linear algebra}); otherwise for some positive integer $m$, the element $\tau:=\tau_2^m\tau_1\in S_0$ 
would have the property that $$\max_{n\ge 1}\dim\ker(\tau-\id)^n$$ is smaller than $\dim B_2$ 
(which is minimal among all elements of $S_0$).

We let $B_1$ be a complementary connected algebraic subgroup of $A$ such that $A=B_1+B_2$, and moreover, each element of $S$ induces an endomorphism of $B_1$.  
So, we reduced to the case that each element of $S$ is of the form $T_y\circ \tau$, where $\tau$ acts on $A=B_1+B_2$ as follows:
\begin{itemize}
\item[(i)] $\tau$ restricted to $B_2$ acts unipotently, i.e. there exists some positive integer 
$r_\tau$ such that $(\tau-\id)^{r_\tau}|_{B_2}=0$;
\item[(ii)] for each $\tau\in \cU_0$, the action of $\tau$ on $B_1$ (which by abuse of notation, 
we also denote by $\tau$) has the property that $\tau^n-\id$ is a dominant map for each positive 
integer $n$ (see \eqref{equation no root of unity}).
\end{itemize}

We proceed similarly to the case $S$ is cyclic. Then for each $\sigma_i\in \cU$ (for 
$i=1,\dots, s$), we let $\tau_i\in \cU_0$, $z_i\in B_1$ and $y_i\in B_2$  such that 
$\sigma_i = T_{y_i+z_i}\circ \tau_i$. Note that it may be that $\tau_i=\tau_j$ for some 
$i\ne j$, but this is not relevant for the proof. We let $C_i$ be the algebraic subgroup of 
$B_2$ spanned by $y_i$ and $(\tau_i-\id)(B_2)$ (for each $i=1,\dots, s$). We recall that 
$\beta_i:=(\tau_i-\id)|_{B_2}$ is a nilpotent endomorphism of $B_2$; we let $\cU_1$ be the 
finite set of all $\beta_i$. Finally, we let $C_S$ be the algebraic subgroup of $B_2$ generated 
by all $C_i$.

If the algebraic subgroup $C_S+B_1$  does not equal $A$, then the exact same argument as in 
Lemma~\ref{fibration exists} yields the existence of a non-constant rational map fixed by each 
$\sigma\in S$. Essentially, the projection map $\pi:A\lra A/(B_1+C_S)$ is a non-constant 
morphism with the property that $\pi\circ \sigma = \pi$ for each $\sigma\in S$.

Next assume $C_S+B_1=A$; we will show there exists $x\in A(\Kbar)$ whose orbit under $S$  is 
Zariski dense. The strategy is the same as in the case $S$ is cyclic. We can find algebraic 
points $x_1\in B_1$ and $x_2\in B_2$ such that the $S$-orbit of $x=x_1+x_2$ is 
Zariski dense in  $A$.  
First we choose $x_2\in B_2(\Kbar)$ as in Lemma~\ref{existence of a point 2} with respect to the 
algebraic group endomorphisms $\beta_i$ and the points $y_i$, for $i=1,\dots, s$; hence the Zariski closure of the group generated by $\beta_i(x_2)+y_i$ is $C_i$ for each $i$.

Let $\Gamma$ be the $\End(A)$-module spanned by $x_2, y_1,\dots, y_s, z_1,\dots, z_s$,
which is a finitely generated subgroup of $A(\Kbar)$.  
Then (using Lemma~\ref{isogenies and m}) we choose $x_1\in B_1(\Kbar)$ such that if 
$\psi\in\End(A)$ has the property that $\psi(x_1)\in \Gamma$, then $B_1\subseteq \ker(\psi)$. 
Let $x:=x_1+x_2$; we will prove that $\OO_S(x)$ is Zariski dense in $A$.

Using Facts~\ref{inside a finitely generated group} and \ref{Faltings theorem}, the Zariski closure of $\OO_S(x)$ is a union of finitely many 
cosets $w_j+H_j$ of algebraic subgroups of $A$. 

\begin{lemma}
\label{a good coset}
There exists a coset $w+H$ of an algebraic subgroup appearing as a component of the Zariski closure of $\OO_S(x)$, and there exists a positive integer $N$ such that $w+H$ is  invariant under $\gamma_i^N$ for each $i=1,\dots, s$.
\end{lemma}

\begin{proof}
So, we know that the Zariski closure of $\OO_S(x)$ is the union of cosets of (irreducible) algebraic subgroups $\cup_{i=1}^\ell (w_i +H_i)$. Let $\gamma\in S$. Then, using that $\gamma\left(\OO_S(x)\right)\subseteq \OO_S(x)$, we obtain  
$$\cup_{i=1}^\ell \left(\gamma(w_i)+\gamma(H_i)\right) \subseteq \cup_{i=1}^\ell (w_i +H_i).$$
On the other hand, each $\gamma\in S$ is a dominant endomorphism of $A$, and therefore, for each $i=1,\dots, \ell$, we have $\dim(\gamma(H_i))=\dim(H_i)$. So, that means $\gamma$ permutes the subgroups $H_i$ of maximal dimension appearing above. In particular, there exists a positive integer $N_0$ such that for each $i=1,\dots, s$, the endomorphism $\gamma_i^{N_0}$ fixes each algebraic group $H_i$ of maximal dimension. 


Let $S^{(N_0)}$ be the submonoid of $S$ consisting of all $\gamma^{N_0}$ for $\gamma\in S$. 
Now, let $H$ be one such algebraic group of maximal dimension among the algebraic groups $H_i$ (for $i=1,\dots, \ell$). Let $w_i+H$ with $i=1,\dots, k$ be all the cosets of $H$ appearing as irreducible components of the Zariski closure of $\OO_S(x)$. Then each element $\gamma\in S^{(N_0)}$ induces a map $f_\gamma:\{1,\dots, k\}\lra \{1, \dots, k\}$ given by $f_\gamma(w_i+H)=w_{f_\gamma(i)}+H$; the map is not necessarily bijective. Moreover, we get a homomorphism of monoids $f:S^{(N_0)}\lra F_k$ given by $f(\gamma):=f_\gamma$, where $F_k$ is the monoid of all functions from the set $\{1,\dots, k\}$ into itself. Clearly,  there exists $j\in\{1,\dots, k\}$, and there exists a positive integer $N_1$ such that $f_{\gamma^{N_1}}(j)=j$ for each generator $\gamma\in \{\gamma_1^{N_0},\dots, \gamma_s^{N_0}\}$ of $S^{(N_0)}$. Then Lemma~\ref{a good coset} holds with $N:=N_0\cdot N_1$. 
\end{proof}

Let $w+H$ be one coset as in the conclusion of Lemma~\ref{a good coset}, and let $N$ be the positive integer from the conclusion of Lemma~\ref{a good coset} with respect to the coset $w+H$. We let $S'$ be the submonoid of $S$ generated by $\gamma_i^N$ for $i=1,\dots, s$. Then $w+H$ contains a set of the form $\OO_{S'}(x')$, for some $x'\in \OO_S(x)$; in other words, $w+H$ contains a set of the form
$$\{\gamma_1^{m_1+Nn_1}\cdots \gamma_s^{m_s+Nn_s}(x)\text{ : }n_1,\dots, n_s\ge 0\},$$
for some positive integers $m_1,\dots, m_s$.   

Let then $\pi:A\lra A/H$ be the canonical projection. Then $$\pi\left(\gamma_1^{m_1+Nn_1}\cdots \gamma_s^{m_s+Nn_s}(x)\right)=w$$ 
for all $n_1,\dots, n_s\ge 0$. Restricted on $B_1$, for each group endomorphism $\tau_i$ (for $i=1,\dots, s$), the action on the tangent space of $B_1$ has no eigenvalue which is a root of unity (see (ii) above); hence    
$$\psi_1:=\left(\tau_1^{m_1+N}\tau_2^{m_2}\cdots \tau_s^{m_s}-\tau_1^{m_1}\tau_2^{m_2}\cdots \tau_s^{m_s}\right)|_{B_1}\text{ is an isogeny.}$$
So, we get that $(\pi\circ\psi_1)(x_1)\in \Gamma$. Because of our choice for $x_1$ and the fact that $\psi_1$ is an isogeny on $B_1$, we conclude that $B_1\subseteq \ker(\pi)$ (note also that $B_1$ is connected by our assumption). Thus $B_1\subseteq H$. So, we can view $\pi$ as a group homomorphism $\pi:B_2\lra A/H$ with the property that for each $n_1,\dots, n_s\ge 0$ we have
$$\pi\left(\gamma_1^{m_1+n_1N}\cdots \gamma_s^{m_s+n_sN}-\gamma_1^{m_1}\cdots \gamma_s^{m_s}\right)(x_2)=0.$$
Letting $\gamma ':=\gamma_1^{m_1}\cdots \gamma_s^{m_s}|_{B_2}$, we have that $\pi\circ \gamma '$ is constant (equal to $w$) on each orbit $\OO_{\gamma_i^N}(x_2)$. Then Corollary~\ref{immediate corollary to many iterates} yields that the connected component of the Zariski closure $C_i$ of the cyclic group generated by $(\tau_i-\id)(x_2)+y_i$ is contained in the kernel of $\pi\circ \gamma '$. Since the $C_i's$ generate the algebraic group $C_S$ (and therefore the connected components of the $C_i$'s generate the connected component of $C_S$; see also Fact~\ref{connected fact}), and furthermore, the connected component of $C_S$ contains the connected component of $B_2$, we conclude that $\pi\circ \gamma '$ is identically $0$ on $B_2$. Because $\gamma '$ is an isogeny, we conclude that $B_2\subseteq \ker (\pi)$, and therefore $H=A$ since $H$ contains both $B_1$ and $B_2$. This concludes our proof.
\end{proof}

\end{document}